\documentclass[11pt]{article}
\usepackage{latexsym}
\usepackage{amsmath,amsthm,amsfonts,amssymb}
\usepackage[T1]{fontenc}
\usepackage[utf8]{inputenc}
\usepackage{aeguill}
\usepackage{url}
\usepackage{enumerate}
\usepackage{mdwlist} 
\usepackage{graphicx}
\usepackage[a4paper,textwidth=15cm,textheight=25cm]{geometry}
\usepackage{xspace}
\usepackage{export}
\usepackage{hyperref}
\newtheorem{thm}{Theorem}[section]

\newtheorem*{thm*}{Theorem}
\newtheorem{dfn}[thm]{Definition} 
\newtheorem*{dfn*}{Definition}

\newtheorem{cor}[thm]{Corollary}
\newtheorem*{cor*}{Corollary}

\newtheorem{prop}[thm]{Proposition} 
\newtheorem*{prop*}{Proposition} 
\newtheorem*{properties*}{Properties} 
 
\newtheorem{lem}[thm]{Lemma} 
\newtheorem*{lem*}{Lemma}

\newtheorem*{claim*}{Claim} 
 
\newtheorem*{fact*}{Fact}

\newtheorem*{qst*}{Question}

\newtheorem*{pb*}{Problem}

\theoremstyle{remark}
 
\newtheorem*{algo*}{Algorithm} 
\newtheorem*{rem*}{Remark}
\newtheorem{rem}[thm]{Remark}
\newtheorem*{example*}{Example}
\newtheorem{example}[thm]{Example}
\newenvironment{preuve}[1][Preuve]{\begin{proof}[#1]}{\end{proof}}

\newcounter{numEnonceTmpInterne}
\newenvironment{enonce*}[1]{\theoremstyle{plain}\stepcounter{numEnonceTmpInterne}%
\def\a{enoncetmp\alph{numEnonceTmpInterne}}%
\newtheorem*{\a}{#1}\begin{\a}}{\end{\a}}

\makeatletter
\edef\@tempa#1#2{\def#1{\mathaccent\string"\noexpand\accentclass@#2 }}
\@tempa\rond{017}
\makeatother

\newcommand{\es}{\emptyset}
\renewcommand{\phi}{\varphi} 
\newcommand{\m} {{}^{-1}} 
\newcommand{\eps} {\varepsilon}

\newcommand {\ra} {\rightarrow}

\newcommand{\ie} {i.e.\ }

\newcommand {\calc} {{\mathcal {C}}}

\newcommand {\calg} {{\mathcal {G}}}

\usepackage{ stmaryrd }

\newcommand{\Z}{{\mathbb {Z}}}

\usepackage{pdfsync}
\newcommand{\inc}{\subset}

\begin{document} 

\title{Elementary equivalence vs commensurability for hyperbolic groups}
\author{Vincent Guirardel, Gilbert Levitt, Rizos Sklinos}
\date{}

\maketitle

\begin{abstract}
We study to what extent  torsion-free (Gromov)-hyperbolic groups are elementarily equivalent to their finite index subgroups. In particular, we prove that a hyperbolic limit group 
 either  is a free product of cyclic groups and surface groups, 
or   admits infinitely many subgroups of finite index which are pairwise non elementarily equivalent. 
\end{abstract}

\section{Introduction}

There are many ways in which one may try to classify finitely generated (or finitely presented) groups. 
The most obvious one is up to isomorphism, or to commensurability (having isomorphic subgroups of finite index). 
  More generally, Gromov suggested the looser notions of quasi-isometry and measure equivalence 
(commensurable groups are quasi-isometric and measure equivalent, though the converse is far from true).

 First-order logic provides us with a coarser relation than the relation of isomorphism, where one tries to classify groups up to elementary equivalence. Two 
groups are elementarily equivalent if and only if they satisfy the same first-order sentences. 
 As usual,
one must restrict 
the class of groups under consideration. A representative example is the work of Szmielew \cite{Szmielew} that characterised all abelian groups  (possibly infinitely generated) up to elementary equivalence.   

Our concern here is to investigate the relation between elementary equivalence and commensurability for the special class of torsion-free hyperbolic groups, 
a class of groups almost orthogonal to the class of abelian groups  (hyperbolic will always mean Gromov-hyperbolic, also called word-hyperbolic or $\delta$-hyperbolic). 

The solution of the Tarski problem by Kharlampovich-Myasnikov \cite{KhMy_elementary} and Sela \cite{Sela_diophantine6}  
says that  
 all non-abelian free groups are elementarily equivalent. 
  Even more, both works characterised the finitely generated groups that are elementarily free   (\ie elementarily equivalent to a 
 non-abelian free group). 
Specifically, there exist finitely generated groups which are elementarily free but not free. The simplest examples are fundamental groups of   closed surfaces with Euler characteristic at most $-2$, 
and  more generally  non-abelian free products of such surface groups and cyclic groups.   This  particular family is closed under taking finite index subgroups;
in particular, these groups  
 are  elementarily equivalent to their finite index subgroups.

  On the other hand, if $G,G'$ are two torsion-free hyperbolic groups with no non-trivial cyclic splitting 
(fundamental groups of closed hyperbolic $3$-manifolds, for instance), the situation is rigid:
if $G$ and $G'$ are elementarily equivalent   
  (or just have the same universal theory), then they are isomorphic  (see Proposition 7.1 of \cite{Sela_diophantine7}).
In particular,  being co-Hopfian \cite{Sela_structure}, the group $G$ is not elementarily equivalent to any of its proper finite index subgroups.

 All finitely generated elementarily free groups are torsion-free and hyperbolic (see    \cite[page 713] {Sela_diophantine6}  and \cite{Sela_diophantine1}). 
  It is therefore natural to ask whether all finitely generated elementarily free groups are elementarily equivalent to
their finite index subgroups, and   more generally which torsion-free   
hyperbolic groups 
are elementarily equivalent to their finite index subgroups.

We shall see that the answer to the   first question is very much negative. Since it  is conceivable that there exist hyperbolic groups with no proper finite index subgroups, 
we focus on classes of hyperbolic groups which have many finite index subgroups: limit groups  and cubulable groups. 

Limit groups were first introduced as finitely generated $\omega$-residually free groups. They may be viewed as limits of free groups in the space of marked groups \cite{CG_compactifying}. 
In model theoretic terms, they may be characterized as finitely generated groups having the same universal theory as free groups
 \cite{Remeslennikov_exist_siberian}. 
Some of them, for instance non-abelian free products of cyclic groups and surface groups as above, are elementarily free,
but most of them are not.  
A group is cubulable  (or cubulated) if it acts properly discontinuously and cocompactly on a CAT(0) cubical complex.
Cubulable hyperbolic groups played a prominent role in the proof by Agol and Wise of the virtually Haken conjecture about 
3-dimensional manifolds \cite{Agol_virtual}. 
 
Our main results are the following:

\begin{thm} \label{mainintro}
Let $G$ be a  torsion-free hyperbolic group. If $G$ is a limit group or is cubulable, then either $G$ is a free product of surface groups and cyclic groups,   or it 
has infinitely many normal subgroups of finite index which are all different up to elementary equivalence.
\end{thm}

\begin{cor}
A finitely generated elementarily free group is elementarily equivalent to all its finite index subgroups if and only if it is a free product of surface groups and cyclic groups.
\end{cor}
We also prove:

\begin{thm} \label{mainintro2}
There exists a  
one-ended    
hyperbolic 
limit  group $G$ which is not elementarily free, but  has an elementarily free finite index subgroup. \end{thm}

Among torsion-free hyperbolic groups, Sela \cite{Sela_diophantine7} singled out a special family, which he called elementary prototypes, with the property that two elementary prototypes 
are elementarily equivalent if and only if they are isomorphic. 
Our way of proving Theorem \ref{mainintro} is to  construct non-isomorphic finite index subgroups of $G$ which are elementary prototypes. 
In Theorem \ref {mainintro2} we construct $G$ as the fundamental group of a surface with a socket.
\\

Let us now describe the contents of the paper in a more detailed way. 

In order to make our main results as independent as possible of the deep theory needed to solve Tarski's problem and establish the results of 
\cite{Sela_diophantine7}, we view one-ended prototypes  as groups having no non-injective preretraction (in the sense of \cite{Perin_elementary}), 
and in Section \ref{eleq} we give a direct argument to show that for such groups elementary equivalence is the same as isomorphism.   As the name suggests, a non-injective preretraction yields a retraction and thus a hyperbolic floor \cite{Perin_elementary}.

A one-ended torsion-free hyperbolic group has a canonical splitting over cyclic groups, its JSJ splitting \cite{Sela_structure,Bo_cut, GL3}. 
A special role is played by the quadratically hanging (QH) vertex groups: they are isomorphic to the fundamental groups of compact surfaces, and incident edge groups are boundary subgroups. 

Using a key technical lemma 
proved in Section \ref{pinch}, we show (Proposition \ref{rephr}) that  the existence of 
a non-injective preretraction is equivalent to that of a surface $\Sigma$ associated to a QH vertex group of the 
JSJ decomposition, and a map $p:\pi_1(\Sigma)\to G$ with the following properties: $p$ is a conjugation on every boundary subgroup, 
is not an isomorphism of $\pi_1(\Sigma)$ onto a conjugate, and has non-abelian image. We also show that $p$ cannot exist if the genus of $\Sigma$ is too small (Corollary \ref{mino}).

The main new results are contained in the last two sections. We first consider fundamental groups of surfaces with sockets, 
obtained by enlarging the fundamental group of a surface with boundary by adding roots of generators of boundary subgroups. 
We determine which of these groups are limit groups,   which of them are elementarily free, and which of them are elementary prototypes  
(i.e.\ do not have non-injective preretractions).  In particular, we prove Theorem \ref{mainintro2} by showing that the fundamental group of an orientable 
surface of genus 2 with a socket of order 3 is a limit group which is an elementary prototype but has an elementarily free subgroup of index 3.  
We also study a related family of groups, for which all sockets are identified.

In Section \ref{fis} we prove Theorem \ref{mainintro}, using the fact that 
a one-ended   torsion-free hyperbolic group $G$ is an elementary prototype if its   cyclic JSJ decomposition satisfies the following conditions: all surfaces 
are orientable, the girth of the graph of groups is large compared to the complexity of the surfaces which appear, and edge groups map injectively to the abelianization of $G$. 
This condition about edge groups is the main reason why we restrict to limit groups and cubulable groups.

\paragraph{Acknowledgements}
This work was mainly conducted  during the  2016 Oberwolfach workshop on Model Theory,
 for which the authors acknowledge support and hospitality.
The first author acknowledges support from the Institut universitaire de France.
The second author is grateful to the organizers of the 2015 Workshop on Model Theory and Groups in Istanbul for making him think about the topic studied here. 
The third author was supported by the LABEX MILYON (ANR-10-LABX-0070) of Université de Lyon, within the program "Investissements d'Avenir" (ANR-11-IDEX-0007) operated by the French National Research Agency (ANR).

\section{Preliminaries}

\paragraph{Graphs of groups.}
Graphs of groups are a generalization of free products with amalgamation and HNN-extensions. 
Recall that a graph of groups $\calg$ consists of a graph $\Gamma$, the assignment of a group $G_v$ or $G_e$ to each vertex $v$ or non-oriented edge $e$, 
and injections $G_e\to G_v$ whenever $v$ is the origin of an oriented edge (see \cite {Serre_arbres} for precise definitions); the image of $G_e$ is called an incident edge group of $G_v$. 
This data yields a group $G(\calg)$, the fundamental group of $\calg$, with an action of $G(\calg)$ on a simplicial tree $T$ (the Bass-Serre tree of $\calg$).  We sometimes call the graph of groups $\Gamma$ rather than $\calg$. 

A graph of groups decomposition (or splitting) of a group $G$ is an isomorphism of $G$ with the fundamental group of a graph of groups. 
The vertex and edge groups $G_v$ and $G_e$ are then naturally viewed as subgroups of $G$. We always assume that $G$ is finitely generated 
and that the graph of groups is minimal (no proper connected subgraph of $\Gamma$ carries the whole of $G$); it follows that $\Gamma$ is a finite graph. 
We allow the trivial splitting, with $\Gamma$ a single vertex $v$ with $G_v=G$.

  Killing all vertex groups defines an epimorphism from $G$ to the topological fundamental group of the graph $\Gamma$ 
(a free group, possibly cyclic or trivial).

The (closed) star of a vertex $v$ of $\Gamma$ is the union of all edges containing $v$. The open star of $v$ is obtained by removing  all vertices $w\ne v$ from the closed star.

\paragraph{Grushko decompositions.}
A finitely generated group has a Grushko decomposition: it is isomorphic to   a free product $A_1*\ldots*A_n*\mathbb{F}_m$, where each $A_i$ for $i\leq n$ is non-trivial, 
freely indecomposable and not infinite cyclic, and $\mathbb{F}_m$ is a free group of rank $m$. Moreover, the numbers $n$ and $m$ are unique, as well as the $A_i$'s up to conjugacy (in $G$) and permutation. 
 We call the $A_i$'s the Grushko factors of $G$. When $G$ is torsion-free, they are precisely the one-ended free factors of $G$   by a classical theorem of Stallings. 
 
Decompositions as free products correspond to graph of groups decompositions with trivial edge groups. 
A one-ended group does not split over a trivial group, but it may split over $\Z$. We discuss this when $G$ is a torsion-free hyperbolic group. 
 
 \paragraph{Hyperbolic groups and their cyclic JSJ decompositions.}
Recall that  a finitely generated group is hyperbolic if its  Cayley graph (with respect to some, hence any, finite generating set) 
is a hyperbolic metric space: there exists $\delta>0$ such that any point on one side of any geodesic triangle is $\delta$-close to one of the other two sides. 
Small cancellation groups and fundamental groups of negatively curved closed manifolds are hyperbolic. In particular
 free groups, fundamental groups of closed surfaces with negative Euler characteristic, and free products of such groups, are hyperbolic. 
A one-ended torsion-free hyperbolic group $G$ is co-Hopfian \cite{Sela_structure}: it cannot be isomorphic to a proper subgroup.
 
A one-ended torsion-free hyperbolic group $G$ has a canonical JSJ decomposition $\Gamma_{can}$ over cyclic groups \cite{Sela_structure, Bo_cut, GL3}. 
We mention the properties that will be important for this paper. 
The graph $\Gamma_{can}$ is bipartite, with every edge joining a vertex carrying a cyclic group to a vertex carrying a non-cyclic group. 
The action of $G$ on the associated Bass-Serre tree $T$ is invariant under 
automorphisms of $G$ and   acylindrical in the following strong sense: if a non-trivial element $g\in G$ 
fixes a segment of length $\ge2$ in $T$, then this segment has length exactly $2$ and its midpoint has cyclic stabilizer. 
 
\paragraph{Surface-type vertices.}
There are two kinds of vertices of $\Gamma_{can}$ carrying a non-cyclic group: rigid ones, and quadratically hanging (QH) ones. 
We will be concerned mostly with QH vertices. 

If $v$ is a QH vertex of $\Gamma_{can}$, the group $G_v$ is isomorphic to the fundamental group of a compact (possibly non-orientable) surface $\Sigma$. 
Incident edge groups $G_e$ are boundary subgroups of $\pi_1(\Sigma)$: there exists a component $C$ of $\partial \Sigma$ such that the image of $G_e$ is equal to $\pi_1(C)$. 
We do not specify base points, so we should really say that $G_e$ is conjugate to $\pi_1(C)$ in $\pi_1(\Sigma)$;   
similarly, any non null-homotopic simple closed curve in $\Sigma$ determines an infinite cyclic subgroup of $\pi_1(\Sigma)$, well-defined up to conjugacy.

Moreover, given any component $C$ of $\partial \Sigma$, there is a unique incident edge $e$ such that $G_e$ equals $\pi_1(C)$. 
These properties are not true for general QH vertices (see  \cite{GL3}), so as in \cite{Perin_elementary} we recall them by calling QH vertices $v$ of $\Gamma_{can}$ \emph{surface-type vertices} (see Definition \ref{stype}).
 
\paragraph{Surfaces.}
 
The genus $g=g(\Sigma)$ of a compact surface $\Sigma$ is the  largest number of nonintersecting (possibly one-sided) simple closed curves (other than components of $\partial \Sigma$) 
that can be drawn on the surface without disconnecting it. If $\Sigma_1,\dots, \Sigma_k$ are disjoint compact subsurfaces of 
$\Sigma$, one clearly has $\sum_k g(\Sigma_i)\le g(\Sigma)$. Surfaces of genus 0 are obtained by removing open discs from a sphere, and  they are called planar.

If $\Sigma$ has $b$ boundary components, its Euler characteristic is $\chi(\Sigma)=2-2g-b$ if $\Sigma$ is orientable, $\chi(\Sigma)=2-g-b$ otherwise. 
Surfaces appearing in surface-type vertices will always have non-abelian fundamental group; equivalently, $\chi(\Sigma)$ will be negative. 
Two surfaces are homeomorphic if and only if they are both orientable or non-orientable, and $g$ and $\chi$ are the same. 
 
Given a compact  surface $\Sigma$, there is an upper bound for the cardinality of a family of disjoint non-parallel simple closed curves 
which are not null-homotopic (a curve is null-homotopic if it bounds a disc, two curves are parallel if they bound an annulus); for instance, 
the bound is $3g-3$ if $\Sigma$ is a closed orientable surface of genus $g\ge2$. 
 
There are $5$ surfaces with $\chi(\Sigma)=-1$. The orientable ones are the pair of pants ($g=0$, $b=3$) and the once-punctured torus ($g=1$, $b=1$). 
The non-orientable ones are the twice punctured projective plane ($g=1$, $b=2$), the once-punctured Klein bottle ($g=2$, $b=1$), and the closed surface of genus 3. 
With the exception of the punctured torus, these surfaces do not carry pseudo-Anosov diffeomorphisms and are not allowed to  appear in hyperbolic towers (in the sense of   \cite{Sela_diophantine1,Perin_elementary}).
They should be considered as exceptional (see Definition \ref{exsur}).

 \paragraph{Actions of surface groups on trees.}
 
   Let $\Sigma$ be a compact surface. 
We describe  a standard construction associating a finite 
family $\calc$ of disjoint simple closed curves on $\Sigma$ to an action of $\pi_1(\Sigma)$ on a tree $T$ such that boundary subgroups are elliptic (they fix a vertex in $T$).
  
The group $\pi_1(\Sigma)$ acts freely on the universal covering $\tilde \Sigma$. It also acts on $T$, and we construct an equivariant  continuous map $\tilde f:\tilde \Sigma\to T$.
  
First suppose that $\Sigma$ is closed. Fix a triangulation of $\Sigma$ and lift it to the universal covering $\tilde \Sigma$. We define $\tilde f$ on the set of vertices, 
making sure that vertices $\tilde \Sigma$ are mapped to vertices of $T$. We then extend it to edges in  a linear way, and to triangles (2-simplices). 
There is a lot of freedom in this construction, but we make sure that $\tilde f$ is in general position with respect to midpoints of edges: 
the preimage $\tilde \calc$ of the set of midpoints of edges of $T$ intersects each triangle in a finite collection of disjoint arcs joining one side of the triangle to another side. 
The construction is the same if $\Sigma$ has a boundary, but we require that each line in $\partial\tilde\Sigma$ 
be mapped to a single vertex of $T$ (this is possible because boundary subgroups act elliptically   on $T$).
  
The subset $\tilde\calc\inc\tilde\Sigma$ is $\pi_1(\Sigma)$-invariant and its projection to $\Sigma$ is a finite family $\calc$ of disjoint simple closed curves.

\paragraph{Limit groups.}

A group $G$ is residually free if, for every nontrivial element $g\in G$, there exists a morphism $h:G\rightarrow \mathbb{F}$ for some free group $\mathbb{F}$ such that $h(g)$ is nontrivial. 

A group $G$ is $\omega$-residually free if, for every finite set    
$\{g_1,\ldots, g_n\}\subset G\setminus\{1\}$, there exists a morphism $h:G\rightarrow\mathbb{F}$ for some free group $\mathbb{F}$ such that $h(g_i)$ is nontrivial for all $i\leq n$.

Remeslenikov \cite{Remeslennikov_exist_siberian} proved that the class of $\omega$-residually free groups coincides with the class of $\forall$-free groups, i.e.\ the class of groups that have the same universal theory as a free group. 
In his work on the Tarski problem, Sela viewed finitely generated $\omega$-residually free groups in a more geometric way and called them limit groups because they arise from limiting processes
 (see \cite{Sela_diophantine1,CG_compactifying}). Limit groups have the same universal theory as a free group, but they are not necessarily elementarily free.

Free abelian groups and free groups are limit groups. Fundamental groups of orientable closed surfaces, and of non-orientable surfaces with $g\ge4$, are limit groups (they are even elementarily free).  
A finitely generated subgroup of a limit group, and a free product of limit groups, are limit groups. A limit group is hyperbolic if and only if it does not contain $\Z^2$ (\cite{Sela_diophantine1}).

\section{Pinched  curves on surfaces}  \label{pinch}

\begin{dfn} [Surface-type vertex]\label{stype}
Let $\Gamma$ be a graph of groups decomposition of a group $G$.  
A vertex $v$ of $\Gamma$  
is called a \emph{surface-type vertex} if the following conditions hold:
\begin{itemize}
 \item the group $G_v$ carried by $v$ is the fundamental group of a compact   surface $\Sigma$ 	  (usually with boundary), with $\pi_1(\Sigma)$ non-abelian;
  \item incident edge groups are maximal boundary subgroups of $\pi_1(\Sigma)$, and this induces a bijection
 between the set of boundary components of $\Sigma$ and the set of incident edges.
\end{itemize}

We say that $G_v$ is a surface-type vertex group. If $u$ is any lift of $v$ to 
the Bass-Serre tree $T$ of $\Gamma$, we say that  $u$  
is a surface-type vertex, and its stabilizer $G_u$ (which is conjugate to $G_v$) is a surface-type vertex stabilizer.
 \end{dfn}
 
 Surface-type vertices are QH (quadratically hanging) vertices in the sense of \cite{GL3}.

\begin{dfn}[Pinching]\label{pin}
  Let $\Sigma$ be a compact surface. 
Given a homomorphism $p:\pi_1(\Sigma)\to G$, a \emph{family of pinched curves} is a collection $\calc$  of disjoint, non-parallel, two-sided simple closed curves $C_i\inc\Sigma$, none of which is null-homotopic, such that the  fundamental group of each $C_i$  is contained in $\ker p$ (the curves may be parallel to a boundary component).

  The map $p$ is \emph {non-pinching} if there is no pinched curve: $p$ is injective in restriction to the fundamental group of any   simple closed curve which is not null-homotopic. 
\end{dfn}

  Let   $S$  be a component of the surface $\hat \Sigma$ obtained by cutting $\Sigma$ along $\calc$.
Its fundamental group is naturally identified with a subgroup of  $\pi_1(\Sigma)$,  so
  $p$ restricts to a map from $\pi_1(S)$ to $G$, which we also denote by $p$. 
   
\begin{lem} \label{lemcle}

Let $\Gamma$ be a graph of groups decomposition of a group $G$, with Bass-Serre tree $T$.
Assume that $\Gamma$ has a single surface-type vertex $v$   with $G_v=\pi_1(\Sigma)$, together with  vertices $v_1,\dots,v_n$ (with $n\ge 1$), and   every edge of $\Gamma$ joins $v$ to some   $v_i$ (in particular,   edge groups of $\Gamma$ are infinite cyclic).

Let $p:\pi_1(\Sigma)\to G$ be a homomorphism such that the image of every boundary subgroup of $\pi_1(\Sigma)$ 
 fixes an edge of  $T$, 
and $p$ is not an isomorphism onto some subgroup of $G$ conjugate to $\pi_1(\Sigma)$. 

Let  $\calc$ be a maximal family of pinched curves on $\Sigma$, and let $S$ be a component of the surface obtained by cutting $\Sigma$ along $\calc$.

 Then the image of   $\pi_1(S) $ by $p$ is contained in a conjugate of some $G_{v_i}$ (\ie it fixes a   vertex of $T$ which is not a lift of $v$).
\end{lem}

\begin{proof} 
The boundary of $S$ consists of components of $\partial \Sigma$ and curves coming from $\calc$. Define $S_0$ by  gluing  a disc   to $S$
 along each curve coming from $\calc$.
 There is an induced map $\pi:\pi_1(S_0)\to G$, hence an action of $\pi_1(S_0)$ on the Bass-Serre tree $T$, and we must show that $\pi_1(S_0)$   fixes a lift of some $v_i$. We may assume that   the image of $\pi$ is non-trivial.

  As in the preliminaries, we consider the universal covering $\tilde S_0$, and $\pi_1(S_0)$-equivariant maps $\tilde f:\tilde S_0\to T$ such that every line in $\partial \tilde S_0$ is mapped to a vertex (such maps exist because boundary subgroups of $\pi_1(\Sigma)$ fix a point in $T$). Assuming that $\tilde f$ is in general position, we consider preimages of midpoints of edges and we project to $S_0$. This yields a finite family $\calc_0$ of simple closed curves on $S_0$, and we choose $\tilde f$ so as to minimize  the number of these curves  (in particular, no curve in $\calc_0$ is null-homotopic). 
 
 The map $\tilde f$ induces a map $f:S_0\to \Gamma$ sending each component $Y$ of $S_0\setminus \calc_0$ into the open star of some vertex $v_Y$ of $\Gamma$. We shall show that   
 \emph{no $v_Y$ may be equal to $v$}. Assuming this, we deduce that $\calc_0$ has to be empty, since every curve in $\calc_0$ separates a component mapped to the star of $v$ from a component mapped to the star of some $v_i$, and the lemma follows.
 
 Let therefore $Y$ be a component  with  $v_Y=v$, and $Z$ its closure. The components of $\partial Z$ come from either $\calc_0$ or $\partial \Sigma$. The restriction $\pi_Z$ of $\pi$ to $\pi_1(Z)$ has an image contained in a conjugate of $\pi_1(\Sigma)$, which we may assume to be $\pi_1(\Sigma)$ itself. It is non-pinching (its kernel cannot contain the  fundamental group of an essential simple closed curve) by maximality of $\calc$, and it sends any  boundary subgroup $H$ of $\pi_1(Z)$ into a boundary subgroup of $\pi_1(\Sigma)$ because $\pi_Z(H)$ fixes an edge of $T$. By Lemma 6.2 of \cite{Perin_elementary}, there are three possibilities: either  $\partial Z=\es$, or the image of $\pi_Z$   is contained in a boundary subgroup of $\pi_1(\Sigma)$, or it has finite index in $\pi_1(\Sigma)$. We show that each possibility contradicts one of our assumptions.
 
If $\partial Z$ is empty   (so, in particular, $\calc_0=\es$ and $Z=S_0$), we note that $\pi_1(\Sigma)$ is a free group because $n\ge1$, and we represent $\pi_Z:\pi_1(Z)\to\pi_1(\Sigma)$ by a map from $Z$ to a graph. Maximality of $\calc$   implies that
 $Z=S_0$ is a sphere or a projective plane, so  the image of $\pi$ is trivial, a contradiction.  

Now suppose that the image of $\pi_Z$ is contained in a boundary subgroup $H$ of $\pi_1(\Sigma)$. 
 Denote by  $\Tilde v$ the vertex 
of $T$ fixed by $\pi_1(\Sigma)$, and by $ \tilde v\tilde v_i$ the 
 unique edge incident on $\Tilde v$ fixed by $H$ (a lift of some edge $vv_i$ of $\Gamma$).
 Consider the component $\tilde Y$ of the preimage of $Y$ in $\tilde S_0$ mapping to the star of $\tilde v$. The union of all regions adjacent to $\tilde Y$ is mapped by $\tilde f$ into the open   star of $\tilde v_i$,
 and we may redefine $\tilde f$ 
so as to remove $\calc_0\cap\partial Z$ from $\calc_0$, thus contradicting the original choice of $\tilde f$ (if $\calc_0\cap\partial Z=\es$, then $Z=S_0$ and   we may replace $\tilde f$ by the constant map equal to $\tilde v_i$).
 
If   the image of $\pi_Z$    has finite index $d$ in $\pi_1(\Sigma)$, then $d$ has to be $1$ because $\chi(\Sigma)\ge\chi(S_0)\ge\chi(Z)=d\chi(\Sigma)$, and we conclude from Lemma 3.12 of \cite{Perin_elementary} that $\Sigma=S_0=Z$ and $p$ is an isomorphism from $\pi_1(\Sigma)$ to itself,  contradicting our hypothesis.
\end{proof}

\begin{rem}    The proof shows the following more general result, which will be useful in \cite{GLS2}: \emph{Let $\Gamma$ and $\Sigma$ be as in Lemma \ref{lemcle}. Let $\Sigma'$ be a compact surface, and let $p:\pi_1(\Sigma')\to  G$  be a homomorphism  such that the image of every boundary subgroup of $\pi_1(\Sigma')$   is contained in a conjugate of some $G_{v_i}$.
Let  $\calc$ be a maximal family of pinched curves on $\Sigma'$, and let $S$ be a component of the surface obtained by cutting $\Sigma'$ along $\calc$. Then, up to conjugacy in $G$, the image of $\pi_1(S)$ by $p$ is contained in some $G_{v_i}$, or there is a subsurface $Z\inc S$ such that $p(\pi_1(Z))$ is a finite index subgroup of $\pi_1(\Sigma)$.
}
 
\end{rem}
\begin{cor} \label{mino}
Let   $\Gamma$ be as in Lemma \ref{lemcle}. Denote by $n$ the total number of vertices $v_i$ of $\Gamma$,  
by  $n_1$ the number of   $v_i$'s which have valence 1, and by $b$ the number of boundary components of $\Sigma$.

If  $p:\pi_1(\Sigma)\to G$ maps every boundary subgroup of $\pi_1(\Sigma)$ injectively into a $G$-conjugate, and $p$ is not  an isomorphism onto a $G$-conjugate of $\pi_1(\Sigma)$, then 
 the genus $g$ of $\Sigma$ satisfies  $g\ge n_1$ and $g+b\ge 2n$.
\end{cor}

Recall that $g=\frac12(2-\chi(\Sigma)+b)$ if $\Sigma$ is orientable, and $g=2-\chi(\Sigma)+b$ otherwise. By a $G$-conjugate of $H$, we mean a subgroup of $G$ conjugate to $H$.

\begin{proof}

We fix a maximal family of pinched curves $\calc\inc\Sigma$, and we let $\hat \Sigma$ be the (possibly disconnected) surface   obtained by cutting $\Sigma$ along $\calc$.

Consider a vertex $v_i$ of $\Gamma$ of valence 1.  Let $C$ be the boundary component of $\Sigma$ associated to the edge $vv_i$, and let $S$ be the component of $\hat\Sigma$ containing $C$. By Lemma \ref{lemcle}, the image of $\pi_1(S)$ by the restriction of   
$p$  fixes a vertex of $T$ which is not a lift of $v$. Because of our assumption on $p(\pi_1(C))$, this vertex must be in the orbit of $v_i$. It follows that $C$ is the only boundary component of $\Sigma$ contained in $S$: since $v_i$ has valence 1, the group $G_{v_i}$ cannot intersect non-trivially a conjugate of $\pi_1(C')$ if $C'\ne C$ is another component of $\partial  \Sigma$. The boundary of $S$ thus consists of $C$ and curves from $\calc$. Since $\pi_1(C)$ is not contained in $\ker p$ but curves in $\calc$ are pinched, $S$ cannot be planar:
  it cannot be a sphere with holes since a generator  of $\pi_1(C)$ would then be a product of  elements representing pinched curves, 
contradicting     injectivity of $p$  on $\pi_1(C)$.

We conclude that $\hat \Sigma$ has at least  $n_1$ non-planar components. This implies $g\ge n_1$.
 The second inequality follows since  $n_1+b\ge 2n$.
\end{proof}

\section{Preretractions}

 \begin{dfn}[JSJ-like decomposition, {\cite[Definition 5.8]{Perin_elementary}}]
 Let $\Gamma$ be a graph of groups decomposition of a group $G$.  It is \emph{JSJ-like} if:

\begin{itemize}
 \item edge groups are infinite cyclic;
 \item at most one vertex of any given  edge is a surface-type vertex (in the sense of Definition \ref{stype}), and at most one carries a cyclic group;
 \item (Acylindricity) if a non-trivial element of $G$ fixes two distinct edges of $T$, then they are adjacent and their common vertex has cyclic stabilizer.
\end{itemize}
\end{dfn}

 The basic example of a JSJ-like decomposition is 
 the canonical cyclic JSJ decomposition of a torsion-free one-ended hyperbolic group $G$.
 \  More generally, the tree of cylinders  (in the sense of \cite{GL4}) of any   splitting of  a torsion-free CSA group $G$   over infinite cyclic   groups is
 JSJ-like.

 \begin{prop}[{\cite [Proposition 6.1]{Perin_elementary}}] \label {isomgr}

Let $\Gamma$ be a JSJ-like graph of groups decomposition of a group $G$. 
If $r:G\to G$ sends each vertex group  and each edge group isomorphically to a conjugate of itself, then $r$ is an automorphism.
\end{prop}
 
This statement is stronger than  Proposition 6.1 of \cite{Perin_elementary}, but the proof is identical.
 If $G$ is torsion-free hyperbolic, $\Gamma$ may be any  graph of groups decomposition with infinite cyclic edge groups (one applies the proposition to its tree of cylinders).

 \begin{dfn}[Exceptional surfaces] \label{exsur}
 
 Pairs of pants, once-punctured Klein bottles, twice-punctured projective planes,  and closed non-orientable surfaces of genus 3, are surfaces with Euler characteristic -1 which do not carry pseudo-Anosov diffeomorphisms. We call them, as well as surface-type vertices associated to them, \emph{exceptional}.
 \end{dfn}

\begin{dfn}[Preretraction, {\cite[Definition 5.9]{Perin_elementary}}] \label{prere}
Let $\Gamma$ be a JSJ-like decomposition of a group $G$. A \emph{preretraction associated to $\Gamma$} is a homomorphism $r:G\to G$ such that, for each vertex group $G_v$:
\begin{enumerate}
\item if $G_v$ is not surface-type,  
the restriction of $r$ to  $G_v$   
is   conjugation  by some $g_v\in G$;
\item if $G_v$ is an exceptional surface-type vertex group, the restriction of $r$ to $G_v$  
is   conjugation by some $g_v\in G$;
\item
if $G_v$ is a non-exceptional surface-type vertex group,   then the   image of $G_v$ is non-abelian.
\end{enumerate}

Note that  this implies that the restriction of $r$ to any edge group of $\Gamma$ is a conjugation. 
\end{dfn}

 \begin{example}  
 Let $G$ be the fundamental group of a
non-exceptional closed surface. Any map from $G$  onto  a non-abelian free subgroup   is a non-injective preretraction (with $\Gamma$ the trivial splitting).
\end{example}

Conditions 2 and 3 are important
in order to draw conclusions from the existence of a non-injective preretraction (as in \cite{Perin_elementary}).
 In the present paper, however, we often do not need these assumptions.

\begin{dfn}[Weak preretraction]\label{weak}
  A map $r:G\to G$ satisfying      
 Condition 1  of Definition \ref{prere} will be called a \emph{weak preretraction}. 
\end{dfn}

Proposition \ref{isomgr} implies that a preretraction $r$ is an isomorphism if there is no non-exceptional surface-type vertex. More generally:  

\begin{prop} \label{rephr}
 Let $\Gamma$ be a JSJ-like decomposition of a   
group $G$. The following are equivalent:
\begin{itemize}
\item
there is a non-injective preretraction $r$ associated to $\Gamma$;
\item    $\Gamma$ has  a non-exceptional surface-type vertex group $G_v=\pi_1(\Sigma)$,
  and there is a map $p:\pi_1(\Sigma)\to G$ which is a conjugation on each boundary subgroup, has non-abelian image, and is not an isomorphism of $\pi_1(\Sigma)$ onto a conjugate.
\end{itemize}

Similarly, there exists  a non-injective \emph{weak} preretraction $r$ if and only if there is an  arbitrary surface-type $G_v$ with a map $p$ which   is a conjugation on each boundary subgroup  and is not an isomorphism of $\pi_1(\Sigma)$ onto a conjugate.
\end{prop}

\begin{proof}
 The existence of $r$ implies that of $p$ by Proposition \ref{isomgr}. To prove the converse direction, we may collapse  the edges of $\Gamma$ not containing the surface-type vertex $v$ associated to $\Sigma$ and thus assume that $\Gamma$ is as in Lemma \ref{lemcle}  (unless $G$ is a non-exceptional  closed surface group,   in which case $p$ itself is a  preretraction which is non-injective because $G$ is co-Hopfian).  As in Lemma \ref{lemcle}, we denote by $n$ the number of vertices adjacent to $v$.
 
  Given $p$ which is a conjugation on each boundary subgroup, we may extend it ``by the identity'' to a (weak) preretraction $r$. We first illustrate this when $\Sigma$ has two boundary components  (so $n\le 2$). 

If 
 $n=2$, we have $G=A_1*_{C_1}\pi_1(\Sigma)*_{C_2}A_2$, and $p$   equals   conjugation by some $g_i$ on the cyclic group $C_i$; we then define $r$ as being 
 conjugation by $g_i$ on $A_i$. If $n=1$, then $G$  
  has presentation $\langle  \pi_1(\Sigma),A,t\mid c_1=a_1, tc_2t\m=a_2\rangle$, with $a_i\in A$ and $c_i$ generators of boundary subgroups $C_i$ of $\pi_1(\Sigma)$. If $p_i$ is conjugation by $g_i$ on $C_i$, we define $r$ as being conjugation by $g_1$ on $A$ and mapping $t $ to $g_1tg_2\m$.

 The general case is similar; if $\Gamma$ is not a tree,  one  chooses a maximal subtree  in order to get a presentation with stable letters as in the case $n=1$ above.

There remains to  check that $r$ is non-injective. 
We may assume that $p$ is injective. Then there is no pinching, so by Lemma \ref{lemcle} the image of $p$ is contained in a conjugate of some $G_{v_i}$. Let  $\tilde e=\tilde v\tilde v_i\inc T$ be a lift of   the edge $vv_i$. We may assume that $\pi_1(\Sigma)$ is the stabilizer of $\tilde v$, and $p$ is the identity on the stabilizer of $\tilde e$.   The extension $r$  then  
preserves the stabilizer of $\tilde v_i$, and up to postcomposing $r$ with the conjugation by some element of $G_{\Tilde e}$, we
may assume that $r$ is the identity on the stabilizer of $\Tilde v_i$. The image of $p$ fixes a vertex in the orbit of $\tilde v_i$, and 
by acylindricity this vertex must be 
  $\tilde v_i$. It follows that $r$ cannot be injective. 
\end{proof}

\begin{rem} 
 In particular, if there is a non-injective preretraction associated to an arbitrary   non-trivial JSJ-like decomposition, there is one associated with a $\Gamma$ as in Lemma \ref{lemcle}.
\end{rem}

\section{Preretractions and elementary equivalence} \label{eleq}
Let $G$ be a finitely generated group. It may be written as $A_1*\cdots*A_n*F$, with $A_i$ non-trivial, not isomorphic to $\Z$, freely indecomposable, and $F$ free. We call the $A_i$'s the Grushko factors of $G$.

 The following is folklore, but we give a proof for completeness.
 
\begin{prop}\label{plgr}
 Let $G$ and $G'$ be torsion-free hyperbolic groups. Assume that $G$ and $G'$ are elementarily equivalent. If $A$ is a Grushko factor of $G$, and $A$ has no non-injective preretraction associated to its cyclic  JSJ decomposition   $\Gamma_{can}$,
 then $A$ embeds into $G'$. 
\end{prop}

Such an $A$ is an example of an elementary prototype \cite[Definition 7.3]{Sela_diophantine7}.

\begin{proof}
 We assume that $A$ does not embed into $G'$, and we construct a non-injective preretraction. By Sela's shortening argument (see \cite [Proposition 4.3]{Perin_elementary}), there exist finitely many non-trivial elements $r_i\in A$ such that, given any homomorphism $f:A\to G'$ (necessarily non-injective),
   there is a modular automorphism $\sigma$ of $A$ such that the kernel of  $f'=f\circ\sigma$ contains one of the $r_i$'s. 

 Let $\Gamma_{can}$ be the canonical cyclic JSJ decomposition of $A$.
 By definition, 
 modular automorphisms act on non-surface vertex groups $G_v$ of $\Gamma_{can}$ by conjugation, so $f$ and $f'$ differ by an inner automorphism on such vertex groups $G_v$.
 If $\Gamma$ has an exceptional surface-type vertex group $G_u$ that has a Dehn twist of infinite order
 (i.e.\ if    $\Sigma$  is a once-punctured Klein bottle or a closed non-orientable surface of genus 3),
this property does not hold for $G_u$ (this detail was overlooked in \cite[Lemma 5.6]{Perin_elementary}).

To remedy this, one considers $\Gamma'_{can}$ obtained from $\Gamma_{can}$ by splitting $G_u$ along the fundamental group of 
a suitable     
2-sided simple closed curve  which is invariant under the mapping class group  
(see Remark 9.32  in \cite{GL3});
in $\Gamma'_{can}$, all exceptional surfaces have finite mapping class group,
and modular automorphisms of $A$ act by conjugation on all exceptional surface groups of $\Gamma'_{can}$.
 
Thus the following statement, which is expressible in first-order logic, holds (compare Section 5.4 of \cite{Perin_elementary}): 
  \emph{given any $f:A\to G'$ such that non-exceptional surface-type vertex groups $G_v$ of  $\Gamma'_{can}$ have non-abelian image, there is $f':A\to G'$ such that $\ker f' $ contains one of the $r_i$'s, non-exceptional  surface-type vertex groups $G_v$  have non-abelian image by $f'$, 
and for each other vertex group $G_v$  of $\Gamma'_{can}$ there is  $g\in G'$   
such that  $f'$ agrees  with  $i_g\circ f$  on $G_v$} (with $i_g$ denoting conjugation by $g$). 
 
 Since $G$ and $G'$ are elementarily equivalent, this  statement holds with $f$ the inclusion from $A$ to $G$, and yields a non-injective $f':A\to G$.   Composing with  a projection from $G$ onto $A$ yields 
 a non-injective weak preretraction from $A$ to $A$  associated to $\Gamma'_{can}$. 
In order to get a preretraction (i.e.\  to ensure that non-exceptional surface-type vertex groups have non-abelian image), we
  apply Proposition 5.12 of \cite{Perin_elementary}.
 By Proposition \ref{rephr}, this preretraction also gives a non-injective preretraction associated to $\Gamma_{can}$. 
\end{proof}

\begin{cor} \label{gelfree}
 If a  finitely generated group 
  $G$ is elementarily free, all its Grushko factors have non-injective preretractions associated to their cyclic  JSJ decompositions.
 \qed
\end{cor}

Recall that every  finitely generated elementarily free group is torsion-free and hyperbolic 
  (\cite {Sela_diophantine1}, \cite{Sela_diophantine6}). 

\begin{cor}
 Let $G$ and $G'$ be one-ended  torsion-free hyperbolic groups. Suppose that $G,G'$   do not have non-injective preretractions associated to their cyclic  JSJ decompositions. 
If $G$ and $G'$ are elementarily equivalent,   they are isomorphic.
\end{cor}

This follows from the proposition and the fact that $G$ and $G'$ are co-Hopfian \cite{Sela_structure}.

\section{Surfaces with sockets}

The fundamental group $G=\pi_1(\Sigma)$ of a closed hyperbolic surface $\Sigma$ is a limit group which is elementarily equivalent to a non-abelian free group, with one exception: if $\Sigma$ is the non-orientable surface of genus 3,   
then $G$ is not  even  a limit group.

In this section we consider fundamental groups of surfaces with sockets (which we simply call socket groups). They are obtained from $\pi_1(\Sigma)$, with $\Sigma$ a compact hyperbolic surface of genus $g$ with $b\ge1$ boundary components $B_i$, by adding a root to the  element $h_i$ representing $B_i$   (see Figure \ref{sock1}). More precisely, $G$ has presentation
$$\langle h_1,\dots,h_b, z_1,\dots,z_b,a_1,b_1),\dots,a_g,b_g\mid h_1\cdots h_b=[a_1,b_1]\cdots[a_g,b_g], z_i^{n_i}=h_i
\rangle\quad $$
or
$$\langle h_1,\dots,h_b, z_1,\dots,z_b,a_1, \dots,a_g \mid h_1\cdots h_b=a_1^2 \cdots a_g^2, z_i^{n_i}=h_i
\rangle,$$
depending on whether $\Sigma$ is orientable or not. The exponent $n_i$, the \emph{order} of the $i$-th socket, is a positive integer, and we always assume $n_i\ge3$: if $n_i=2$, we can remove the socket by attaching a M\"obius band to $B_i$.  

\begin{figure}[ht!] 
\centering
\includegraphics[width=.8\textwidth]{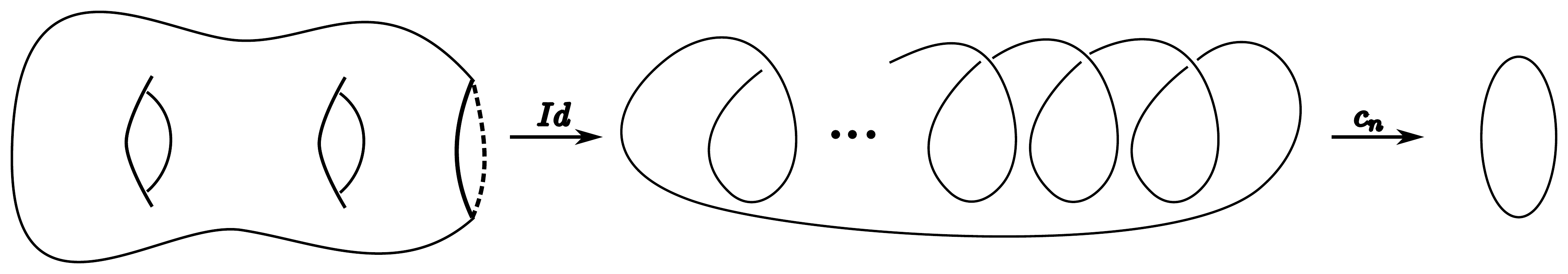}
\caption{A socket group can be seen topologically as the fundamental group of a surface for which every boundary component is glued onto a circle under an $n$-cover map for some $n$.}
\label{sock1}
\end{figure}

Socket groups are torsion-free one-ended hyperbolic groups. They were introduced  in  \cite{Sela_structure} to study splittings over maximal cyclic subgroups.

We shall show:
\begin{prop} \label{soclim} 
 The group $G$ is a limit group if and only if one of the following holds:
\begin{itemize}
\item the surface $\Sigma$ is orientable with at least four boundary components (i.e. $b\geq 4$); 
\item the surface $\Sigma$ is orientable with two or three boundary components  and genus  $g\geq 1$;
\item the surface $\Sigma$ is orientable with one boundary component and genus $g\geq \frac{n+1}2$, where $n$ is the order of the unique socket;
\item the surface $\Sigma$ is non-orientable and $b+g \geq 4$.
\end{itemize}
\end{prop}

\begin{prop} \label{socelfree}
The following assertions are equivalent: 
\begin{enumerate}
\item the group $G$ is elementarily equivalent to a free group;
\item the groups $G$ has a non-injective preretraction associated to its cyclic JSJ decomposition;
\item  the surface $\Sigma$ is non-orientable, $b+g\ge4$, every $n_i$ is even, and $g\ge b$.
\end{enumerate}

\end{prop}

The group $G$ is naturally the fundamental group of a graph of groups $\Gamma$ with a central vertex carrying $\pi_1(\Sigma)$, and $b$ terminal vertices carrying $\langle z_i\rangle$ (the edge groups are generated by the $h_i$'s).
It is a torsion-free one-ended hyperbolic group, and the assumption $n_i\ge3$ ensures that $\Gamma$ is its canonical cyclic JSJ decomposition (see  Proposition 4.24   
in \cite{DG2}).

\begin{proof}[Proof of Proposition \ref {soclim}]
By Proposition 4.21 of \cite{CG_compactifying}, $G$ is a limit group if and only if  the equation 
 $$z_1^{n_1}\cdots z_b^{n_b}=[a_1,b_1]\cdots[a_g,b_g]  \qquad \text{or} \qquad 
 z_1^{n_1}\cdots z_b^{n_b}=a_1^2 \cdots a_g^2   ,$$  
with unknowns $z_i,a_i,b_i$  in a free group $F(x,y)$,  has a solution which is not contained in a cyclic subgroup, 
and with every $z_i$ non-trivial.
 
Recall that, given integers $k_i\ge2$,  a relation   $x_1^{k_1}\cdots x_p^{k_p}=1$  between elements of $F(x,y)$ implies that $\langle x_1,\dots,x_p\rangle$ is cyclic  
when $p\le 3$   \cite{LySc_equation}. 
 Thus $G$ is not a limit group if $\Sigma$ is non-orientable and $b+g\le3$. On the other hand, if $\Sigma$ is non-orientable and $b+g\ge4$, 
it is easy to find a suitable solution of the equation above, by setting all unknowns equal to powers of $x$ or $y$.

Now  suppose that $\Sigma$ is orientable and there is a single socket ($b=1$). If $g\ge\frac{n+1}2$, it follows from \cite{Culler_equations}   
that $[x,y]^n$ is a product of $g$ commutators, so $G$ is   a limit group. If $g<\frac{n+1}2$, no non-trivial $n$-th power is a product of $g$ 
commutators by \cite{ComEd_genus}, so $G$ is not a limit group.

Finally, assume that $\Sigma$ is orientable and $b\ge2$.
As suggested by Jim Howie, the equation $z_1^{n_1} z_2^{n_2}=[a_1,b_1]$ is solved by setting 
$z_1=x^{n_2}$, $z_2=yx^{-n_1}y\m$, $a_1=x^{n_1n_2}$,   $b_1=y$. 
It follows that $G$ is   a limit group when   $g\ge1$. If $g=0$, the fact recalled above implies that $G$ is   a limit group if and only if $b\ge4$.
\end{proof}

\begin{proof}[Proof of Proposition \ref {socelfree}]
We first prove $(2)\implies(3)$.
Assuming that there is a preretraction, we apply Lemma \ref {lemcle} to the map $p$ provided by Proposition \ref{rephr}. Cut $\Sigma$ open along $\calc$, and call 
$S_i$   the component of the surface thus obtained  which contains $B_i$. 
For each $i$, up to conjugation, we get a map from $\pi_1(S_i)$ to $\langle z_i\rangle$ sending $h_i$ to itself  
and killing   the fundamental group of every other boundary component of $S_i$. 
This is possible only if  $S_i$ is non-orientable and $n_i$ is even. 

The inequality $g\ge b$ follows from Corollary \ref{mino}.   
By Proposition \ref{isomgr}, the existence of a non-injective preretraction forces that of a non-exceptional surface. This  rules out the once-punctured Klein bottle ($g=2$, $b=1$) 
and thus implies $b+g\ge4$.

Conversely, we assume that $\Sigma$ and the $n_i$'s are as in   (3), and we construct a non-injective preretraction $r$ fixing every $z_i$ (hence every $h_i$). Write $n_i=2m_i$.
If $b=1$, then $g\geq 3$;  we choose $z\in G$ not commuting with $z_1$, and we map    $a_1$ to $z_1^{m_1}$, $a_2$ to $z$, $a_3$ to $z\m$, and $a_i$ to 1 for $i>3$. 
If $2\le b\le g$, we map $a_i$ to $z_i^{m_i}$ for $i\le b$, to 1 for $i>b$.  This proves $(3)\implies(2)$.

$(1)\implies(2)$ follows from Corollary \ref{gelfree}. Conversely, the map $r$ constructed above expresses $G$ as an extended hyperbolic tower over a free group, showing that $G$ is elementarily free \cite{Sela_diophantine6}.
\end{proof}

\begin{rem}\label{socelfree3}
The proof shows that  $G$ has a non-injective weak preretraction 
(in the sense of   Definition  \ref{weak}) associated to its cyclic JSJ decomposition  if and only if $\Sigma$ is non-orientable, every $n_i$ is even, and $g\ge b$.
\end{rem}

\begin{cor}
Let $G$ be a socket group, with $\Sigma$ a once-punctured orientable surface of genus 2, and a single socket of order 3
(see Figure \ref{SockInd3}).
Then $G$ is a one-ended hyperbolic limit group $G$ which is not elementarily free but contains an elementarily free subgroup of finite index.
\end{cor}

\begin{rem}
In fact, $G$ has no non-injective weak preretraction by Remark \ref{socelfree3}.
\end{rem}

\begin{proof}
By Propositions   \ref{soclim} and    \ref{socelfree}, $G$ is a limit group that is  not elementarily free.
It has a subgroup $G_0$ of index 3, obtained as the kernel of a map from $G$ to $\Z/3\Z$ killing $a_1,b_1,a_2,b_2$, 
which is the fundamental group of the space obtained by gluing three once-punctured surfaces of genus 2 along their boundary. Identifying two of these surfaces yields a retraction from $G_0$ 
onto the fundamental group of the closed orientable surface of genus 4,  
so $G_0$  is elementarily free by \cite{Sela_diophantine6}.
\end{proof}
 
\begin{figure}[ht!]
\centering
\includegraphics[width=.9\textwidth]{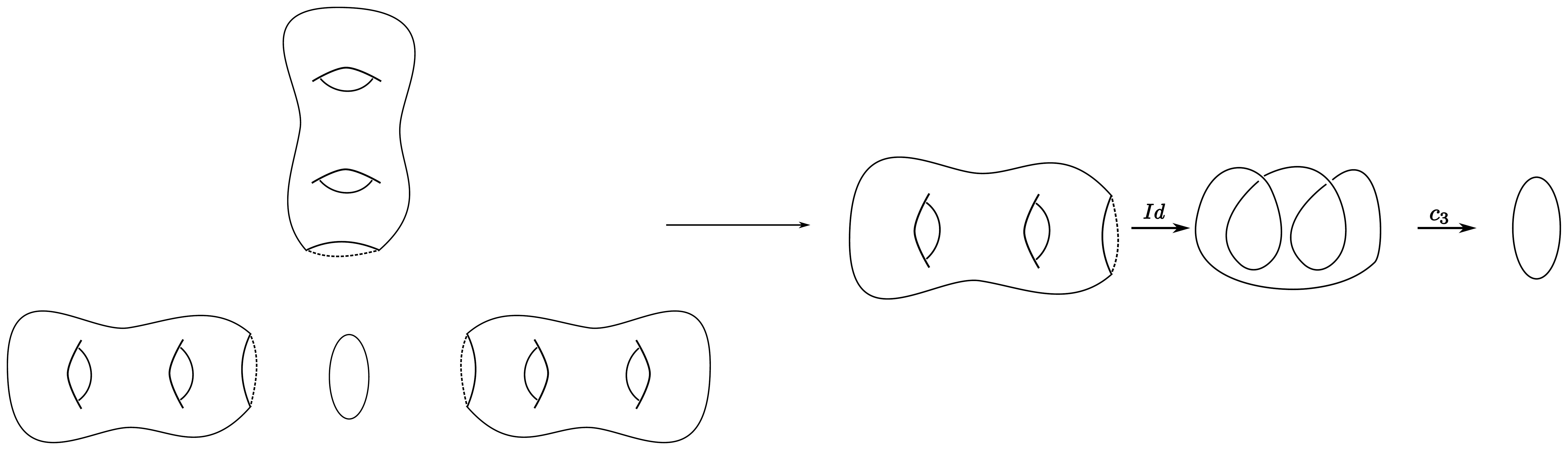}
\caption{A covering of   
  degree $3$.}
\label{SockInd3}
\end{figure}

We now consider a slightly different family of groups, where all sockets are identified. More precisely, 
$G$ is the fundamental group of a graph of groups with one surface-type vertex group $G_v=\pi_1(\Sigma)$,  where $\Sigma$ has at least two boundary components, and one vertex carrying a cyclic group $\langle z \rangle$. It is presented as 

$$\langle h_1,\dots,h_b, z,t_1,\dots,t_b,a_1,b_1,\dots,a_g,b_g\mid h_1\cdots h_b=[a_1,b_1]\cdots[a_g,b_g], t_iz^{n_i}t_i\m=h_i, t_1=1
\rangle\quad $$
or
$$\langle h_1,\dots,h_b, z,t_1,\dots,t_b,a_1, \dots,a_g \mid h_1\cdots h_b=a_1^2 \cdots a_g^2, t_iz^{n_i}t_i\m=h_i, t_1=1
\rangle,$$
depending on whether $\Sigma$ is orientable or not  (see Figure \ref{Amagalm3}). 

\begin{figure}[ht!]
\centering
\includegraphics[width=.3\textwidth]{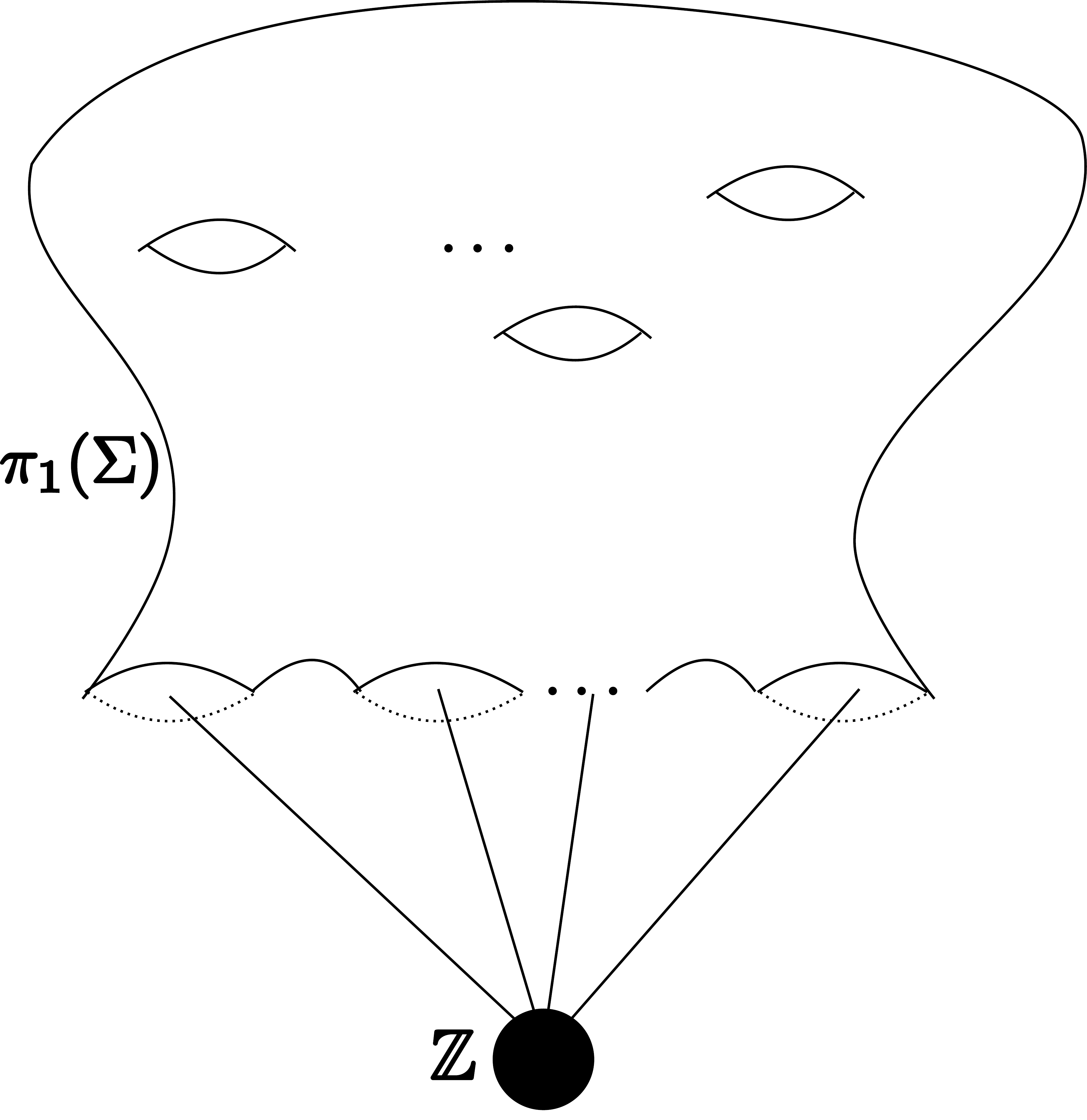}
\caption{A socket group with identified sockets.}
\label{Amagalm3}
\end{figure}

When $\Sigma$ is orientable one may orient the components of $\partial \Sigma$ in a consistent way, 
and the sign of the $n_i$'s matters. In the non-orientable case the components may be oriented separately and only $ | n_i | $ has a meaning.

\begin{prop}
Let $G$ be   a socket group with identified sockets as above.
 There is a non-injective weak preretraction   associated to the cyclic JSJ decomposition of $G$  if and only if the sum $\displaystyle
 \sum _{i=1}^bn_i$ is $0$ when $\Sigma$ is orientable, the sum $\displaystyle\sum_{i=1}^b  n_i$ is even when $\Sigma$ is non-orientable.
\end{prop}

\begin{proof}
 Assume that there is $p:\pi_1(\Sigma)\to G$ which  is a conjugation on each boundary subgroup  and is not an isomorphism of $\pi_1(\Sigma)$ onto a conjugate   (see Proposition \ref{rephr}). We apply Lemma \ref{lemcle}. First suppose there is no pinching. All $h_i$'s are mapped into the same conjugate of $\langle z\rangle$, so we may assume that $h_i$ is mapped to a power of $z$, which must be $z^{n_i}$ (note that $z^p$ conjugate to $z^q$ implies $p=q$). We deduce that  $z^{\sum n_i}$ is a product of squares or commutators in $\langle z\rangle$. If there is pinching, apply the argument to each pinched surface.
 
 Conversely, we define a non-injective preretraction $r$   as follows: we  map $z$ to $z$, every $h_i$ to  $z^{n_i}$, every $t_i, a_i,b_i$ to 1, except that we  map $a_1$  to  $z^{( \sum n_i)/2}$ in the non-orientable case. 
\end{proof}

Deciding whether there is a non-injective (true) preretraction $r$ (with $\Sigma$ non-exceptional and the image of $\pi_1(\Sigma)$ non-abelian) is more subtle.

   First suppose that $\Sigma$ is orientable. If its genus is $\ge1$, one may make    $r(\pi_1(\Sigma))$ non-abelian by sending $a_1$ and $b_1$ to 
some  $x\in G$ which does not commute with $z$.

If  $\Sigma
$ is a 4-punctured sphere and $r$ exists, Lemma \ref{lemcle} implies that there is a pinched curve.  It must divide $\Sigma $ into two pairs of pants, each containing two boundary components of $\Sigma$, say $B_1$ and $B_2$ on one side, $B_3$ and $B_4$ on the other. One must then have $n_1+n_2=n_3+n_4=0$. Conversely, if this holds, 
 one defines $r$ on the generating set   $\{z,t_2,t_3, t_4\}$ by sending $z$ to itself, $t_2$ to the trivial element, and by sending $t_3$ and $t_4$ to 
some  $x\in G$ such that $z$ and $xzx\m$ do not commute. This argument shows:

 \begin{prop}
Let $G$ be a socket group with identified sockets as above, with   $\Sigma$ an orientable surface of genus $g$. 
There is a non-injective preretraction (i.e. $G$ is elementarily free) if and only if the sum $\displaystyle \sum _{i=1}^bn_i$ is $0$ and either:
    \begin{itemize}
     \item[(i)]    
     $g\geq 1$, or 
     \item[(ii)]   
     $g=0$ and there is a proper nonempty subset $I\inc \{1,\dots, b\}$ with $\sum_{i\in I} n_i=0$.      \qed
     \end{itemize}

\end{prop}

When $\Sigma$ is non-orientable, the easy case is when $g\ge3$ since 
one may use $a_2$ and $a_3$ to make $r(\pi_1(\Sigma))$ non-abelian. When $g$ is 1 or 2, 
one has to consider the ways in which a maximal family of pinched curves may divide $\Sigma$. A case by case analysis yields:

\begin{prop}
Let $G$ be a socket group with identified sockets as above, with   $\Sigma$ a non-orientable surface of genus $g$.
Then there is a non-injective preretraction (i.e. $G$ is elementarily free) if and only if the sum  $\displaystyle
 \sum _{i=1}^bn_i$ is even and one of the following holds:
    \begin{itemize}
    \item[(i)]    
    $g\geq 3$;
    \item[(ii)]   $g=2$ and $\Sigma$ has  
at least three boundary components (i.e.\ $b\geq 3$);  
    \item[(iii)]     
    $g=2$,   the surface   $\Sigma$ has two boundary components (i.e.\ $\Sigma$ is a twice punctured Klein bottle), and  
    $ | n_1 | = | n_2 | $ 
     or both $n_1, n_2$ are even;   
    \item[(iv)]  
     $g=1$, the surface $\Sigma$ has   
at least three boundary components (i.e.\ $b\geq 3$), and there exist   a proper subset  $I\inc \{1,\dots, b\}$ and signs $\eps_i=\pm1$ with $\sum_{i\in I} \eps_i n_i=0$.
   \qed
     \end{itemize}
  
\end{prop}

\section{Finite index subgroups} \label{fis}

The most obvious examples of limit groups, namely free abelian groups, free groups, surface groups, are elementarily equivalent to their finite index subgroups. 
The goal of this section is to show that these are basically the only examples of limit groups with this property, at least among hyperbolic limit groups.

\begin{thm}  \label{main}
 Let $G$ be a hyperbolic limit group, or a torsion-free cubulable hyperbolic group. If $G$ is not a free product of cyclic groups and of  surface groups, it has infinitely many normal finite index subgroups which are all different up to elementary equivalence. 
\end{thm}

\begin{cor} \label{elfree}
Let $G$ be a finitely generated elementarily free group. If $G$  is not a free product of cyclic groups and surface groups, it has infinitely many finite index subgroups which are not elementarily free.  
\end{cor}
On the other hand, if $G$ is a free product of cyclic groups and  fundamental  groups of closed surfaces with $\chi(\Sigma)\le-2$,  it is elementarily free    
by \cite{Sela_diophantine6}, and so are all its non-cyclic finitely generated subgroups.

 The corollary is clear, because a finitely generated elementarily free group is a hyperbolic limit group  by  \cite{Sela_diophantine1} and   \cite{Sela_diophantine6}.
We shall deduce the theorem   
from  a more technical statement, which requires a definition.

\begin{dfn}\label{evq}
We  say that $G$ has \emph{enough  abelian virtual quotients}  if, given any infinite cyclic subgroup $C$, there is a finite index subgroup $G_C$ of $G$  such that $C\cap G_C$ has infinite image in the abelianization of $G_C$; equivalently, $C$ has a finite index subgroup which is a retract of a finite index subgroup of $G$.
\end{dfn}

Note that having enough  virtual abelian quotients is inherited by subgroups. 
 This property is slightly weaker than $G$ being 
    LR over cyclic groups in the sense of \cite{LoRe_subgroup} (they require that $C$ itself be a retract of a finite index subgroup). It implies the existence of   infinitely many finite index subgroups.

\begin{thm} \label{vrai}
 Let $G$ be a one-ended  torsion-free hyperbolic  group. If $G$ has enough  virtual abelian quotients and is not a   surface group, then $G$ has infinitely many characteristic subgroups of finite index which 
  have no non-injective weak preretraction associated to their cyclic  JSJ decomposition.
\end{thm}

The theorem readily extends to groups which are only virtually torsion-free and are not virtual surface groups.

Having enough  virtual abelian quotients is a technical assumption (which holds for limit groups and cubulable hyperbolic groups, as we shall see). It is not optimal, but note that some assumption is needed: if there is a hyperbolic group which is not residually finite, there is one which has no proper subgroup of finite index \cite{KaWi_equivalence}, and such a group obviously does not satisfy Theorem \ref{vrai}. Also note that groups with property (T) do not have enough  virtual abelian quotients, but they cannot   have non-injective weak preretractions  
because they have no splittings.

We first explain how to deduce Theorem \ref{main}  
from Theorem \ref{vrai}.
  
\begin{lem} \label{abquot}
 Residually free  groups (in particular limit groups) and cubulable hyperbolic groups have enough  abelian virtual quotients.
\end{lem}

\begin{proof}
If $G$ is residually free, there is a map $p$ from $G$ to a free group which is injective on $C$. By Hall's theorem, 
the image of $C$ under $p$ is a free factor of a finite index subgroup, and we let $G_C$ be its preimage.

If $G$ is hyperbolic and cubulable, then by \cite{Agol_virtual} it is virtually special: 
 there is some finite index subgroup $G_1<G$ acting freely and cocompactly on a CAT(0) cube complex $X$
such that $X/G_1$ is special.
Since quasiconvex subgroups are separable by \cite[Corollary  7.4]{HaWi_special},
up to passing to a further finite index subgroup, we can assume that the  $2$-neighbourhoods of hyperplanes are embedded in $X/G_1$
so that $X/G_1$ is fully clean. Denote   $C_1=C\cap G_1$.

By \cite[Proposition 7.2]{HaWi_special}, there is a convex cube complex $Y\subset X$ which is $C_1$-invariant and $C_1$-cocompact   and  defines a local isometry $f:Y/C_1\ra X/G_1$   which is fully special by \cite[Lemma 6.3]{HaWi_special}. 
By \cite[Proposition 6.5]{HaWi_special}, there is a finite cover $X/G_2$ of $X/G_1$ and an embedding $Y/C_1\ra X/G_2$ with a retraction $X/G_2\ra Y/C_1$. 
Group-theoretically, we get a retraction from the finite index subgroup $G_2$ of $G$ to $C_1$. This proves the lemma.
\end{proof}

  Note that any residually free hyperbolic group is a limit group \cite{Baumslag_residually}.

\begin{proof} [Proof of Theorem \ref{main}]

Given $G$ as in the theorem,
we write its Grushko decomposition as  $G=A_1*\cdots*A_p*F*S_1*\dots*S_q$ 
where $F$ is free, each $S_i$ is a  closed surface group, and each $A_i$ is a one-ended group which is not a surface group 
(the  $A_i$'s are not necessarily distinct up to   isomorphism). By assumption, we have $p\ge1$.

By Theorem \ref{vrai} and Lemma \ref{abquot}, we can find in  each $A_i$  a decreasing sequence  $A_i(1)\supsetneqq A_i(2)\supsetneqq A_i(3)\supsetneqq\dots$ of characteristic  finite index subgroups which   have no non-injective preretraction associated to their cyclic  JSJ decomposition. Let $G_n$ be the kernel of the natural projection from $G$ onto $\prod_{i=1}^p A_i/A_i(n)$. It is a finite index subgroup of $G$ whose non-surface Grushko factors are all isomorphic to some $A_i(n)$.

We claim that there are infinitely many distinct $G_n$'s up to elementary equivalence. If not, we may assume that they are all equivalent. Applying Proposition \ref{plgr} to $G_1$ and $G_2$, we find that $A_1(1)$ embeds into $G_2$, hence into some Grushko factor  $A_j(2)$. We cannot have $j=1$ because $A_1(1)$ properly contains $A_1(2)$  and torsion-free one-ended hyperbolic groups are co-Hopfian \cite{Sela_structure}. We may therefore assume $j=2$. We then find that $A_2(2)$ embeds in some $A_3(k)$, and $k>2$ by co-Hopfianity.
 Iterating this argument leads to a contradiction.
\end{proof}
 
The remainder of this section is devoted to the proof of Theorem \ref{vrai}. 

We consider the canonical cyclic JSJ decomposition $\Gamma_{can}(G)$, or simply  $\Gamma_{can}$,   and the associated Bass-Serre tree $T_{can}$.   
Its non-rigid  vertices are of surface-type.   

Recall that the girth of   a   graph  is the smallest length of an embedded circle ($\infty$ if the graph is a tree). 
We say that $\Gamma_{can}$ has \emph{large girth} if its girth $N$ is large  compared to   the complexity  of the surfaces which appear in $\Gamma_{can}$; precisely, $(N-2)/2$ should be larger than the maximal cardinality of a family of non-parallel disjoint simple closed curves on a given surface (the curves may be boundary parallel, but should not be null-homotopic).

If $G_0$ is a finite index subgroup, its JSJ tree is obtained by restricting  the action of $G$ on  
$T_{can}$ to $G_0$. This is because by \cite{Bo_cut} one may construct  $T_{can}$  purely from the topology of the boundary of $G$, and $\partial G_0=\partial G$.   If $v$ is a surface-type vertex for the action of $G$ on $T_{can}$, it is one also for the action of $G_0$, but in general the surface is replaced by a finite cover.

The proof requires two lemmas.

\begin{lem} \label{bongpe}
There is 
  a characteristic  subgroup of finite index $G_0\inc G$ whose JSJ decomposition has   the following properties: 
\begin{enumerate}
\item
the surfaces which appear in $\Gamma_{can}(G_0)$ are all orientable;
\item each edge group maps injectively to the abelianization of $G_0$;
\item the graph $\Gamma_{can}(G_0)$ has large girth (as defined above).
\end{enumerate}
\end{lem}

Note that the first two properties are inherited by finite index subgroups. 

\begin{lem} \label{step2}
Let $G$ be a torsion-free one-ended hyperbolic group whose JSJ decomposition   satisfies the   properties of  the previous lemma.    
If $G$  is not a surface group, then $G$     has no non-injective weak preretraction associated to its cyclic  JSJ decomposition.
\end{lem}

Before proving these lemmas, let us   explain how they imply  Theorem \ref{vrai}. 

First suppose that there is $G_0$ as in Lemma \ref{bongpe} 
 such that $\Gamma_{can}(G_0)$ is not a tree. Then choose a decreasing sequence of characteristic finite index subgroups $  H_1\supset H_2\supset\cdots$ of the topological fundamental group of the graph $\Gamma_{can}(G_0)$, and define $G_n$ as the preimage of $H_n$ under the natural  epimorphism  from $G_0$ to the fundamental group  of  the graph $\Gamma_{can}(G_0)$. 

Since $\Gamma_{can}(G_0)$ is invariant under automorphisms of $G_0$, each $G_n$ is characteristic in $G$. Its JSJ decomposition is the covering of $\Gamma_{can}(G_0)$ associated to $H_n$, with the lifted graph of groups structure.
The group $G_n$ satisfies the first two properties of Lemma \ref{bongpe}, and also the third one because the surfaces appearing in $\Gamma_{can}(G_n)$ are the same as in $\Gamma_{can}(G_0)$, so Lemma \ref{step2} applies.

  If $\Gamma_{can}(G_0)$ is   a tree for every subgroup of finite index $G_0<G$ as in Lemma \ref{bongpe},  we fix any such $G_0$  
 and we let $G_n$ be any sequence of  distinct characteristic subgroups of finite index (it exists because there are enough abelian  virtual quotients). We cannot control the complexity of surfaces, but Lemma \ref{step2} applies to $G_n$ because the girth is   infinite.

We now prove the lemmas.

\begin{proof}[Proof of Lemma \ref{bongpe}]
Since the first two properties are inherited by finite index subgroups, it suffices to construct $G_0$ having one given property.

 To achieve (1), view the JSJ decomposition  $\Gamma_{can}$ as expressing $G$ as the fundamental group of a graph of spaces  $X$, with a surface $\Sigma$ for each surface-type vertex. If there is a non-orientable surface, consider  a 2-sheeted covering of $X$ which is trivial (a product) above the complement of the non-orientable surfaces, and is the orientation covering $\hat \Sigma$ over each non-orientable surface (note that  a boundary component of $\Sigma$ lifts to two curves in $\hat \Sigma$). Though $\Gamma_{can}$ is canonical, the 2-sheeted covering is not, so we
  define $G_0$ as the intersection of all   subgroups of $G$ of index 2.

 For (2), let $C_i$ be the edge groups of $\Gamma_{can}$, and let $G_0$ be any characteristic subgroup of finite index contained in every group $G_{C_i}$ provided by   Definition \ref{evq}.
  Each edge group of $\Gamma_{can}(G_0)$ is the image of some $C_i\cap G_0$ by an automorphism of $G_0$, so maps injectively to the abelianization.

If the girth of $\Gamma_{can}$ is too small, we choose a characteristic finite index subgroup $H_0$ of the topological fundamental group $\pi_1(\Gamma_{can})$ such that the associated finite cover of $\Gamma_{can}$ has large girth, and  we define $G_0$ as the preimage of $H_0$ under the    epimorphism  from $G$ to  the topological fundamental group of $ \Gamma_{can}$. As in the proof of the theorem given above, invariance of $\Gamma_{can}$ under automorphisms implies that $G_0$ is characteristic, and the complexity of surfaces does not change.
\end{proof}

\begin{preuve}[Proof of Lemma \ref{step2}]

We consider the canonical cyclic JSJ decomposition $\Gamma=\Gamma_{can}$. Recall that it is bipartite: each edge joins a vertex carrying an infinite  cyclic group to a vertex carrying a non-cyclic group.

We argue by way of contradiction, assuming that there is a non-injective weak preretraction. By Proposition \ref{rephr}, there is a surface-type vertex group  $G_v=\pi_1(\Sigma)$ and a map $p:\pi_1(\Sigma)\to G$ which is a conjugation on each boundary subgroup and is not an isomorphism of $\pi_1(\Sigma)$ onto a conjugate. 

Let  $\calc$ be a maximal family of pinched curves on $\Sigma$ (see Definition \ref{pin}), and let $S$ be a component of the compact surface obtained by cutting $\Sigma$ along $\calc$. Since $G$ is not a surface group, we may assume that $S$ contains a boundary component $C$ of $\Sigma$.
   It must contain another boundary component $C'$ of $\Sigma$: otherwise, since $\Sigma$ is orientable and curves in $\calc$ are pinched, the image of a generator of $\pi_1(C)$ would be a product of commutators, contradicting the second item in Lemma \ref {bongpe}.
  
     We apply Lemma \ref{lemcle} to the graph of groups obtained from $\Gamma$ by collapsing all edges   which do not contain $v$. 
   We find that the image of $\pi_1(S)$ by $p$ is contained (up to conjugacy) in the fundamental group $R_i$ of a subgraph of groups $\Gamma_i$ of $\Gamma$, which is a component of the complement of the open star of $v$. The image of the fundamental groups of $C$ and $C'$ are contained in the groups carried by distinct vertices $w$ and $w'$ of $\Gamma_i$ which are adjacent to $v$ in $\Gamma$, but far away from each other in $\Gamma_i$ because the girth of $\Gamma$ is large.

As in the proof of Lemma \ref{lemcle},
  we construct a surface $S_0$ by attaching   discs to boundary curves of $S$ coming from $\calc$, and we 
represent   the induced map $\pi:\pi_1(S_0)\to R_i$ by an equivariant map   from the universal covering $\tilde S_0$ to the Bass-Serre tree of $\Gamma_i$. We  consider preimages of midpoints of edges, and we project to $S_0$.
 We obtain a finite family $\calc_0$ of disjoint simple closed curves on $S_0$.
   
 View $\Gamma_i$ as covered by stars of vertices $u$ carrying a cyclic group. If $C_0$ is a curve in $\calc_0$, the image of $\pi_1(C_0)$ by $\pi$ is contained (up to conjugacy) in a unique $G_u$, and   curves   associated to different vertices cannot be parallel. Since any path joining $w$ to $w'$ in $\Gamma_i$ must go through at least $(N-2)/2$ stars, with $N$ the girth of $\Gamma$, the family $\calc_0$ must contain at least $(N-2)/2$ non-parallel curves. We get a contradiction if $N$ is large enough.
 \end{preuve}
 

 \bigskip

\begin{flushleft}
 Vincent Guirardel\\
Institut de Recherche Math\'ematique de Rennes\\
Universit\'e de Rennes 1 et CNRS (UMR 6625)\\
263 avenue du G\'en\'eral Leclerc, CS 74205\\
F-35042  RENNES C\'edex\\
\emph{e-mail: }\texttt{vincent.guirardel@univ-rennes1.fr}\\[8mm]

Gilbert Levitt\\
Laboratoire de Math\'ematiques Nicolas Oresme\\
Universit\'e de Caen et CNRS (UMR 6139)\\
(Pour Shanghai : Normandie Univ, UNICAEN, CNRS, LMNO, 14000 Caen, France)\\
 \emph{e-mail: }\texttt{levitt@unicaen.fr}\\[8mm]

Rizos Sklinos\\ 
Institut Camille Jordan\\
Universit\'{e} Claude Bernard Lyon 1 (UMR 5208)\\ 
43 boulevard du 11 novembre 1918 \\
69622, Villeurbanne cedex, France\\ 
\emph{e-mail: }\texttt{rizozs@gmail.com}\\[8mm]

\end{flushleft}

\end{document}